\documentclass[11pt, oneside]{amsart}

\usepackage{tabularx, hyperref}
\usepackage{amssymb} \usepackage{amsfonts} \usepackage{amsmath}
\usepackage{amsthm} \usepackage{epsfig, subfig}
\usepackage{ amscd, amsxtra, latexsym}
\usepackage[all]{xy}
\usepackage{caption}
\usepackage{enumerate}
\usepackage{color}

\newtheorem{lemma}{Lemma}[section]
\newtheorem{thm}[lemma]{Theorem}
\newtheorem{prop}[lemma]{Proposition}
\newtheorem{cor}[lemma]{Corollary}

\newtheorem{thmintro}{Theorem}

\theoremstyle{definition}
\newtheorem{defn}[lemma]{Definition}

\theoremstyle{definition}

\definecolor{darkgreen}{cmyk}{1,0,1,.2}

\newcommand{\R} {\ensuremath {\mathbb{R}}}

\newcommand{\Z} {\ensuremath {\mathbb{Z}}}

\newcommand{\calC} {\ensuremath {\mathcal{C}}}

\newcommand{\calP} {\ensuremath {\mathcal{P}}}
\newcommand{\calA} {\ensuremath {\mathcal{A}}}
\newcommand{\calB} {\ensuremath {\mathcal{B}}}

\address{Department of Mathematics, ETH Zurich, 8092 Zurich, Switzerland}
\email{sisto@math.ethz.ch}

\begin{document}

\title{Quasi-convexity of hyperbolically embedded subgroups}
\author{Alessandro Sisto}

\begin{abstract}
 We show that any infinite order element $g$ of a virtually cyclic hyperbolically embedded subgroup of a group $G$ is Morse, that is to say any quasi-geodesic connecting points in the cyclic group $C$ generated by $g$ stays close to $C$. This answers a question of Dahmani-Guirardel-Osin. What is more, we show that hyperbolically embedded subgroups are quasi-convex.

Finally, we give a definition of what it means for a collection of subspaces of a metric space to be hyperbolically embedded and we show that axes of pseudo-Anosovs are hyperbolically embedded in Teichm\"uller space endowed with the Weil-Petersson metric.
\end{abstract}

\maketitle

\section{Introduction}

Hyperbolically embedded subgroups have been introduced by Dahmani-Guirardel-Osin in \cite{DGO}, and they give a way to capture common properties of several classes of groups including (relatively) hyperbolic groups, mapping class groups, $Out(F_n)$, many groups acting on $CAT(0)$ spaces and many others.

Osin further showed in \cite{OS-acyl} that the class of groups containing non-degenerate hyperbolically embedded subgroups coincides with other classes of groups that have been defined for the same purpose of providing a common perspective on all groups listed above \cite{BeFu-wpd,Ha-isomhyp, Si-contr}.

Given the extensive list of examples, it is interesting to study the geometry of (word metrics of) groups with hyperbolically embedded subgroups. In particular, \cite[Problems 9.3,9.4]{DGO} are about the existence of Morse elements in groups containing non-degenerate hyperbolically embedded subgroups. An element $g$ is Morse if the map $n\mapsto g^n$ is a quasi-isometric embedding and any $(\mu,c)$-quasi-geodesic whose endpoints are in $\langle g\rangle$ is contained in the neighbourhood $N_C(\langle g\rangle)$, where $C=C(\mu,c)$. The following gives a positive answer to Problem 9.4 and hence Problem 9.3 as well.

\begin{thmintro}\label{main1}
Let $G$ be a finitely generated group.
 If the infinite order element $g\in G$ is contained in a virtually cyclic subgroup $E(g)$ which is hyperbolically embedded in $G$, then $g$ is Morse.
\end{thmintro}

Several special cases of the theorem are known already \cite{Be-asgeommcg, DMS-div, A-K, BC-raagcones}.

In view of the fact that loxodromic elements for acylindrical actions on hyperbolic spaces give rise to hyperbolically embedded subgroups \cite[Theorem 4.42]{DGO}, Theorem \ref{main1} generalises \cite[Theorem 4.9]{DMS-div}. 

We will actually prove the following more general result.

\begin{thmintro}\label{main}
Let $G$ be finitely generated subgroup and let $\{H_i\}$ be a finite family of finitely generated subgroups of $G$.
 If $\{H_i\}\hookrightarrow_h G$, then for each $\mu,c$ there exists $C$ with the property that all $(\mu,c)$-quasi-geodesics in $G$ whose endpoints are in $H_i$ are contained in $N_C(H_i)$.
\end{thmintro}

Theorem \ref{main} was previously known for relatively hyperbolic groups \cite{DS-treegr}.

Of course, the proof requires understanding geometric features of hyperbolically embedded subgroups, and the key one we will show is, roughly speaking, that if $H\hookrightarrow_h (G,X)$ then balls in $Cay(G,X)$ are geometrically separated in the word metric of $G$, see Lemma \ref{fibgeomsep}.

\subsection*{Hyperbolically embedded subspaces?}
The most important geometric question about groups with hyperbolically embedded subgroups is whether this notion is a quasi-isometry invariant (there is actually more than one way to specify what this means).

In order to study this question it is probably useful to have an appropriate notion of what it means for a family of subspaces to be hyperbolically embedded in a metric space. In the case of relatively hyperbolic groups, it has been the development of the theory of relatively hyperbolic metric spaces started in \cite{DS-treegr} that lead to the proof of quasi-isometric rigidity, at first with additional hypotheses in \cite{DS-treegr} and then in general in \cite{D-relhyp}.

Drawing inspiration from the key geometric fact we show about groups with hyperbolically embedded subgroups, Lemma \ref{fibgeomsep}, in Section \ref{specs} we suggest a possible definition of hyperbolic embedded family of subspaces of a metric space. It is straightforward to check that this definition is quasi-isometry invariant using the quasi-isometry invariance of relative hyperbolicity, see Proposition \ref{qiinv}. Analysing this definition, or suitable variations, might shed light on the question of quasi-isometry invariance of hyperbolically embedded subgroups.

A ``naturally occurring'' example of metric space that is not a group and is not relatively hyperbolic but contains hyperbolically embedded subspaces is Teichm\"uller space with the Weil-Petersson metric, the hyperbolically embedded subspaces being axes of pseudo-Anosovs, see Theorem \ref{teich}. For the non-existence of nontrivial relatively hyperbolic structures on the Teichm\"uller space of any sufficiently complicated surface see \cite{BDM-thick, BM-relhypteich}.

\section{Background}
We denote the Cayley graph of a group $G$ with respect to the generating system $Y$ by $Cay(G,Y)$. For $X$ a metric space, we denote its metric by $d_X$. Balls of radius $r$ are denote by $B^X_r(\cdot)$ (with ``$X$'' dropped if unambiguous). Similarly, the (closed) neighbourhood of radius $R$ of the subset $A$ of a metric space is denoted by $N^X_{R}(A)$ or $N_R(A)$.

\begin{defn}
 Let $X$ be a metric space. We say that $A,B\subseteq X$ are \emph{geometrically separated} if for each $D$ the diameter of $N_D(A)\cap B$ is finite.
\end{defn}

The most convenient characterization of hyperbolically embedded subgroups for our purposes is the following, see \cite[Theorem 6.4]{Si-metrrh}.

\begin{defn}
 For $G$ a finitely generated group and $H_1,\dots,H_n<G$ finitely generated subgroups we say that $\{H_i\}$ is \emph{hyperbolically embedded} in $(G,Y)$, where $Y$ is a subset of $G$, if $\Gamma=Cay(G,Y)$ is hyperbolic relative to the collections $\{gH_i\}$ of the left costs of the $H_i's$ and $d_{\Gamma}|_{H_i}$ is proper for each $i$. In such case we write $\{H_i\}\hookrightarrow_h (G,Y)$.
\end{defn}

Relatively hyperbolic spaces are also referred to in the literature as asymptotically tree-graded, most notably in \cite{DS-treegr}.

The only facts we need about relatively hyperbolic spaces are collected in the following lemma. The first item is \cite[Theorem 4.1-$(\alpha_1)$]{DS-treegr}, while (stronger versions of) the other two can be deduced from results in \cite{DS-treegr} and can be found in \cite{Si-proj}.

\begin{lemma}\label{rhfacts}
Let $X$ be a metric space hyperbolic relative to the collection of subsets $\calP$. For $P\in\calP$, denote by $\rho_P:X\to P$ any map satisfying $d(x,\rho_P(x))\leq d(x,P)+1$.
\begin{enumerate}
 \item For each $D$ there exists $B$ so that whenever $P,Q\in \calP$ are distinct the diameter of $N_D(P)\cap Q$ is at most $B$.
 \item There exists $C$ so that for each $P\in\calP$ and $x,y\in X$ we have $d(\rho_P(x),\rho_P(y))\leq Cd(x,y)+C$.
 \item There exists $C$ so that for each $P\in\calP$ if $d(\rho_P(x),\rho_P(y))\geq C$ for some $x,y\in X$ then $d(x,y)\geq d(x,P)$.
\end{enumerate}
\end{lemma}

\section{Geometric separation of fibres in $G$}

In general, when $H\hookrightarrow_h (G,Y)$ the action of $G$ on $Cay(G,Y)$ need not be acylindrical. Osin shows that if one is allowed to change $Y$ then one can make the action acylindrical, but we will not need this fact and we introduce the following definition instead.

\begin{defn}
Let $G$ be a group acting on the metric space $X$ and let $H\subseteq X$ a subspace. We say that the action is \emph{acylindrical along} $H$ if for every $r$ there exists $R$ so that for any $x,y\in H$ with $d(x,y)\geq R$ there are only finitely many $g\in G$ so that $d(x,gx),d(y,gy)\leq r$.
\end{defn}

It has to be mentioned that the weak proper discontinuity property of Bestvina and Fujiwara \cite{BeFu-wpd} is very similar to the acylindricity condition we gave above.

First of all, let us show that the definition applies to the context we are interested in.

\begin{lemma}
 If $\{H_i\}_{i=1,\dots,n}\hookrightarrow_h (G,Y)$ then the action of $G$ on $\Gamma=Cay(G,Y)$ is acylindrical along each $H_i$.
\end{lemma}

\begin{proof}
This follows easily from the geometric separation of the cosets of $H_i$ in $\Gamma$ (Lemma \ref{rhfacts}-(1)), which implies that if $g$ is as in the definition of acylindricity along $H_i$ and $R$ is large enough, then $g$ actually lies in $H_i$. The conclusion follows from the fact that there are finitely many elements of $H_i$ in any given ball of radius $r$, as the restriction of $d_\Gamma$ to $H_i$ is proper.
\end{proof}

The key fact we use to prove Theorem \ref{main1} is the following.

\begin{lemma}\label{fibgeomsep}
Let $X$ be a metric space and $H\subseteq X$ a subspace.
 Suppose that the group $G$ acts acylindrically on $X$ along $H$ and let $\pi$ be the orbit map of the basepoint $x_0\in X$. Then for each $r$ there exists $R$ so that if $x,y\in X$ satisfy $d(x,y)\geq R$ then $\pi^{-1}(B^X_r(x))$ and $\pi^{-1}(B^X_r(y))$ are geometrically separated in a Cayley graph of $G$.
 \end{lemma}

\begin{proof}
In this proof we denote $d_X$ by $d$, and similarly balls are taken in $X$ unless otherwise stated. A word metric on $G$ will be denote d $d_G$.

Given $r$, choose $R$ so that there are finitely many elements $h\in G$ so that $d(x,hx), d(y,hy)\leq 2r$ whenever $d(x,y)\geq R-2r$ and $x,y\in H$.

 In the notation of the lemma, we say that $g\in G$ is a \emph{bridge} with \emph{anchor point} $a\in G$ if $\pi(a)\in B_r(x)$, $\pi(ag)\in B_r(y)$.

 \begin{figure}[h]
  \includegraphics{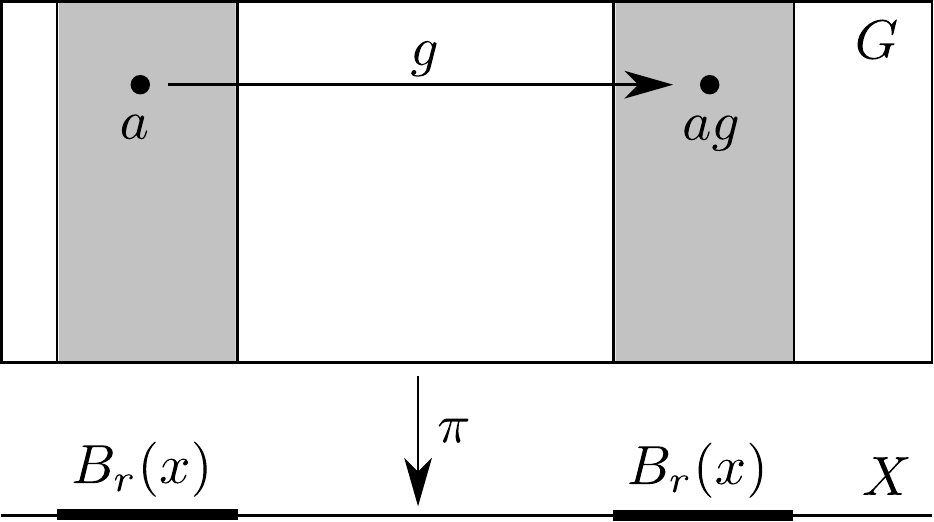}
  \caption{$g$ is a bridge with anchor point $a$.}
\end{figure}
 
 We claim that for any $g\in G$ there are finitely many elements $a$ so that $g$ is a bridge with anchor point $a$. As there are finitely many elements in a given ball of radius, say, $D$ in a given Cayley graph of $G$, assuming the claim we conclude that for each $D$ we can find $K$ so that if $\pi(a)\in B_r(x)$, $\pi(b)\in B_r(y)$ (so that $a^{-1}b$ is a bridge) and $d_G(a,1),d_G(1,b)>K$ then $d(a,b)>D$.
This is the same as saying that $\pi^{-1}(B_r(x))$ and $\pi^{-1}(B_r(y))$ are geometrically separated in a Cayley graph of $G$.
 
 We are left to prove the claim. Suppose that $a_1$ and $a_2$ are anchor points for $g\in G$. Then $d(a_1^{-1}a_2x_0, x_0)\leq 2r$, as $\pi(a_i)\in B_r(x)$ means by definition that $a_ix_0$ is in $B_r(x)$. Similarly
 $$d(a_1^{-1}a_2 (g x_0),gx_0)=d(a_2g x_0, a_1g x_0)\leq 2r,$$
as $a_igx_0\in B_r(y)$.
 As $d(x_0, gx_0)=d(a_1x_0,a_1gx_0)\geq R-2r$, we see that $a_1^{-1}a_2$ is one of finitely many possible elements by our choice of $R$. This proves the claim. 
\end{proof}


The statement we actually need is the following.

\begin{lemma}\label{propersep}
  Suppose that the group $G$ acts acylindrically on the metric space $X$ and let $\pi$ be the orbit map of the basepoint $x_0\in X$. Let $H$ be a subset of $G$ so that $\pi|_H$ is proper. Then for each $x\in H$ and $r\geq 0$ there exists $R$ so that if $y\in H$ satisfies $d(x,y)\geq R$ then the following holds. For each $D$ there exists $B$ \emph{not depending on $y$} so that $N^G_D(\pi^{-1}(B^X_r(x)))\cap \pi^{-1}(B^X_r(y))\subseteq N_B(H)$.
\end{lemma}


\begin{proof}
We can choose $R$ as in Lemma \ref{fibgeomsep} with $r$ replaced by $2r$. Let $M$ be a Lipschitz constant for $\pi$. For a fixed $D$, by compactness of $W=B^X_{MD+2r}(x)$ it is easy to see that we can find $B$ so that the required property holds for each $y\in W$, as one can find a finite family of balls of radius $2r$ so that any ball of radius $r$ with centre in $W$ is contained in a ball in the family. On the other hand, if $y\notin W$, then there cannot be $g,h$ with $\pi(g)\in B^X_r(x),\pi(h)\in B^X_r(y)$ and $d_G(x,y)\leq D$.
\end{proof}

\section{Superlinear divergence}

We fix the notation of the theorem and denote throughout the section $X=Cay(G,Y)$ for $Y\subseteq G$ so that $\{H_i\}\hookrightarrow_h (G,Y)$. To slightly simplify the notation, we denote $H=H_1$.
Also, for $h\in H$ and $r_1\leq r_2$ we set $A(h,r_1,r_2)=H\cap (B^X_{r_2}(h)\backslash \mathring{B}^X_{r_1}(g))$ and $S(h,r_1)=A(h,r_1,r_1)$.

All paths we consider are discrete, meaning that their domain is an interval in $\mathbb Z$. The length of the path $\alpha$ with domain $[m,n]$ is just
$$l(\alpha)=\sum_{m\leq i<n} d(\alpha(i),\alpha(i+1)).$$

From now on, $d$ denotes $d_X$, while we always use the subscript $d_G$ when referring to a fixed word metric on $G$.

\subsection{The idea of the proof}

Let us illustrate the idea of the proof in the case when $H$ is cyclic.
Suppose we have a Lipschitz path $\alpha$ in $G$ connecting points in $H$. Our aim is to show that $\alpha$ is very long if it strays far from $H$. Let us regard $\alpha$ as  a path in $X$. We fix sufficiently far away ``checkpoints'' along $H$ between the endpoints of $\alpha$ and ask ourselves whether $\alpha$ intersects balls of sufficiently large radius around the checkpoints.
Avoiding one such ball ``costs'' a lot of length, because $X$ is hyperbolic, so a subpath of $\alpha$ that avoids a checkpoint ball is very long. In particular, if there are many checkpoint balls that are not intersected by $\alpha$, then $\alpha$ is very long, as required. Otherwise, there are many consecutive pairs of checkpoint balls both intersected by $\alpha$. In this case we will use geometric separation with respect to $d_G$ to say that points in the intersection of the given balls and $\alpha$ are far away in $G$, and in particular a subpath of $\alpha$ connecting them has to be very long.

\begin{figure}[h]
 \includegraphics{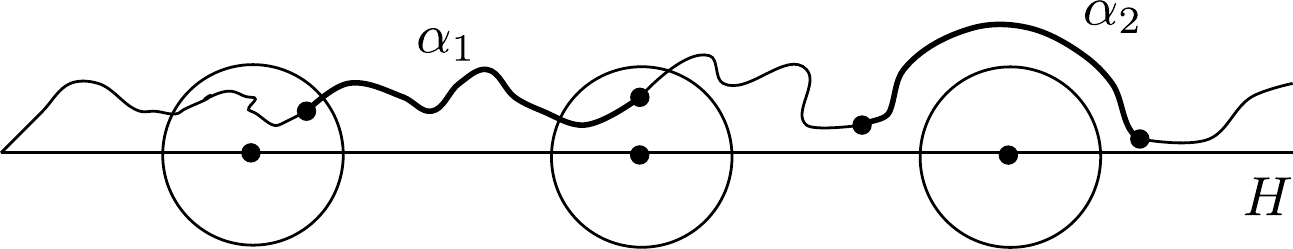}
 \caption{$\alpha_1$ is long because of geometric separation, while $\alpha_2$ is long because it is a detour in a hyperbolic space.}
\end{figure}

Some care has to be taken in choosing the constants involved in the construction, and actually we will be forced to consider sufficiently far apart balls instead of consecutive balls.
Also, in the general case we replace checkpoints by annuli. The way we ``keep track of $\alpha$'' along $H$ is via a closest point projection to $H$.

\subsection{The proof}

The following is a form of superlinear divergence for $H$.

\begin{prop}\label{suplin}
 For each $L$ there exists $K$ with the following property. Let $\alpha$ be a 1-Lipschitz path in $G$ from $h_1\in H$ to $h_2\in H$ that avoids $N^G_K(A(h_1,r_1,r_2))$, where $0<r_1<r_2<d(h_1,h_2)$. Then
 $$l(\alpha)\geq L(r_2-r_1).$$
\end{prop}

We now deduce Theorem \ref{main} from the proposition. If $H$ was cyclic, the conclusion would follow more quickly from results in \cite{DMS-div}.

\emph{Proof of Theorem \ref{main}.}
 Consider a $(\mu',c')$-quasi-geodesic $\beta'$ in $G$ from $1$ to $h\in H$. For some $\mu=\mu(\mu',c'),c=c(\mu',c')$ there exists a 1-Lipschitz path $\beta:(I\subseteq \Z)\to G$ with endpoints $1,h$, whose image is within Hausdorff distance $c$ from the image of $\beta'$ and so that for each $x,y\in I$ we have
$$l(\beta|_{[x,y]})\leq \mu d(\beta(x),\beta(y))+c.\ \ \ \ \ (1)$$

Let $K$ be given by the proposition setting $L=\mu+1$. Suppose that $\beta(I)$ is not contained in $N^G_K(H)$, for otherwise we are done, and let $x,y\in I$ be so that $\beta|_{[x,y]}$ intersects $N^G_K(H)$ only at its endpoints. Our goal is now to give a bound $B$ on $l(\beta|_{[x,y]})$ depending on $\mu,c,H$, which in turn gives us $\beta(I)\subseteq N^G_{K+B}(H)$, as required.

Consider the path $\alpha$ obtained concatenating (in the appropriate order) $\beta|_{[x,y]}$ and geodesics of length at most $K$ connecting $x$, $y$ to some $h_1,h_2\in H$ respectively. We can assume $d(h_1,h_2)>2K$, for otherwise it is easy to find a bound on $d(\beta(x),\beta(y))$ and hence, using (1), on $l(\beta|_{[x,y]})$. We can set $r_1=K$, $r_2=d(h_1,h_2)-K$ and apply the proposition to get
$$l(\beta|_{[x,y]})\geq l(\alpha)-2K\geq Ld(h_1,h_2)-2KL-2K\geq$$
$$(\mu+1)d(\beta(x),\beta(y))-2K(\mu+1)-2KL-2K.$$
This contradicts (1) if $d(\beta(x),\beta(y))$ is sufficiently large. So, we get a bound on $d(\beta(x),\beta(y))$ and in turn a bound on $l(\beta|_{[x,y]})$ again using (1).\qed

\medskip

We are left to prove the proposition.

\medskip

\emph{Proof of Proposition \ref{suplin}.} Fix from now on $L\geq 0$.
Let $\rho:X\to H$ be a map satisfying, for each $x\in X$, $d(x,\rho(x))=d(x,H).$
(The minimum exists because $d|_H$ is proper.)

By Lemma \ref{rhfacts}-(2), the map $\rho$ is coarsely Lipschitz, meaning that it satisfies $d(\rho(x),\rho(y))\leq C_1 d(x,y)+C_1$ for a suitable constant $C_1$.

\begin{lemma}
 There exists $C_2\geq 10 C_1+1$ with the following property. If $\beta$ is a 1-Lipschitz path in $X$ from $x$ to $y$ and there exist $h\in H, r>0$ so that $N_{C_2}(S(x,r))\cap \beta=\emptyset$, $d(h,\rho(x))<r-C_2$ and $d(h,\rho(y))>r+C_2$ then $l(\beta)\geq 10^3 C_2 L$.
\end{lemma}

\begin{proof}
 Lemma \ref{rhfacts}-(3) says that there exists $C$ so that if $d(\rho(x),\rho(y))\geq C$ then $d(x,y)\geq d(x,H)$. 
 Now, if we choose any $C_2\geq 100 C_1+100C$, using the fact that $\rho$ is coarsely Lipschitz we see that there exist points $p_1,\dots,p_n$ on $\beta$ (appearing in the given order) and a constant $D=D(C,C_1)$ so that
 \begin{enumerate}
  \item $n\geq C_2/D + 1$,
  \item $d(\rho(p_i),\rho(p_{i+1}))\geq C$,
  \item $\rho(p_i)\in  N_{C_2/2}(S(x,r))$.
 \end{enumerate}

As $N_{C_2}(S(x,r))\cap \beta=\emptyset$, from $3)$ we see $d(p_i, H)\geq C_2/2$. Then, using 2) we have $d(p_i,p_{i+1})\geq C_2/2$. Finally, the length of $\beta$ is at least $\sum d(p_i,p_{i+1})$ and hence at least $C_2^2/(2D)$. To conclude the proof, we are only left to pick $C_2$ large enough.
\end{proof}

Fix $C_2$ as in the lemma, and assume $r_2-r_1\geq \max\{10^4 C_2 R\}$, where $R\geq 1$ is an integer as in Lemma \ref{propersep} with $r=C_2$. We can make this assumption because the distance between $h_1,h_2$ is always at least $K$, so that $l(\beta)\geq K$, and hence we can use $K$ large enough so that the proposition holds for $r_2-r_1\leq \max\{10^4 C_2 R\}$.

Consider radii $\{r_i\}_{i=1,\dots,n}$ with the following properties.
\begin{enumerate}
 \item $r_i\in [r_1,r_2]$,
 \item $r_{i+1}-r_i\geq 3C_2$,
 \item $n\geq (r_2-r_1)/(10C_2)$,
\end{enumerate}

Consider the sets $T_i=N_{C_2}(S(h_1,r_i))$. If $\alpha\cap T_i=\emptyset$ then we can take a subpath $\alpha_i$ of $\alpha$ so that the lemma applies to $\beta=\alpha_i$ (except if $i=1$ or $i=n$). All such subpaths can be chosen to be disjoint because of condition 2). In particular, if the set of $i$'s so that $\alpha\cap T_i=\emptyset$ has cardinality at least $(r_2-r_1)/(100C_2)$ then the length of $\beta$ is larger than $(r_2-r_1)L$, as required.
Suppose that this is not the case. Then subdividing the integers up to $n$ according to their remainder modulo $R$ and using a simple counting argument gives that there exists $s$ so that there are at least $(r_2-r_1)/(100C_2R)$ indices $i$ so that $\beta\cap T_{iR+s},\beta\cap T_{(i+1)R+s}\neq\emptyset$. In such case pick any $p_i\in \beta\cap T_{iR+s}, p_{i+1}\in \beta\cap T_{(i+1)R+s}$.
By Lemma \ref{propersep} with $x=\rho(p_i)$ and $y=\rho(p_{i+1})$, if we take $K$ large enough we get $d(p_i,p_{i+1})\geq 100C_2 RL$.
More in detail, in our case $\pi$ is the identity map, $r=C_2$, $H$ is our given hyperbolically embedded subgroup, $x=\rho(p_i)$, $y=\rho(p_{i+1})$ and $D=100C_2 RL$. Notice that $d(x,y)\geq 3C_2R-2C_2\geq R$, and we have $p_i\in \pi^{-1}(B_{C_2}(x))$, $p_{i+1}\in \pi^{-1}(B_{C_2}(y))$.
Also, the constant $B$ from Lemma \ref{propersep} can be chose to be independent of $x$ because we can use the action of $H$ to assume $x=1$. Now, if $d(p_i,p_{i+1})$ was at most $100C_2 RL$ then $p_i$ would be at distance at most $B$ from a point $h\in H$.
Such $h$ would be at distance at most $B+C_2$ from $\rho(p_i)$ in $X$, which would in turn imply, as distance measured between points in $H$ are comparable when measured in $X$ or in $G$, that we 
can uniformly bound $d_G(p_i,\rho(p_i))$. However, $\rho(p_i)\in A(h_1,r_1,r_2)$, so we get a contradiction with the hypothesis that $\beta$ avoids $N^G_K(A(h_1,r_1,r_2))$ if $K$ is large enough.

Clearly, $\sum d(p_i,p_{i+1})$ is a lower bound for $l(\alpha)$, and hence $l(\alpha)\geq L(r_2-r_1)$, as required.\qed

\section{Towards hyperbolically embedded metric spaces?}\label{specs}
As mentioned in the introduction, in this section we give a tentative definition of what it means for a collection of subspaces of a metric space to be hyperbolically embedded. First, we give a preliminary definition.

\begin{defn}(cfr. Lemma \ref{fibgeomsep})
 Let $X,Y$ be metric spaces, $H\subseteq Y$ and let $\pi:X\to Y$ be a map. We say that $\pi$ is acylindrical along $H$ is for each $r$ there exists $R$ so that $\pi^{-1}(B^Y_r(x))$ is geometrically separated from $\pi^{-1}(B^Y_r(y))$ whenever $x,y\in H$ satisfy $d(x,y)\geq R$.
\end{defn}

Here is a the definition of hyperbolic embeddability for metric space. The author makes no claim that this is the ``right'' one. A map $f:X\to Y$, for $X,Y$ metric spaces, is coarsely Lipschitz (resp. coarsely surjective) if there exists a constant $C$ so that $d_Y(f(x),f(y))\leq Cd(x,y)+C$ for any $x,y\in X$ (resp. $N^Y_C(f(X))=Y$).

\begin{defn}
 Let $X$ be a geodesic metric space, $\calP$ a collection of subsets and $\pi:X\to Y$ a coarsely Lipschitz and coarsely surjective map, with $Y$ another geodesic metric space. We say that $\calP$ is hyperbolically embedded in $(X,\pi)$, and write $\calP\hookrightarrow_h (X,\pi)$, if $Y$ is hyperbolic relative to $\{\pi(P)\}_{P\in\calP}$, $\pi|_P$ is proper and $\pi$ is acylindrical along $\pi(P)$ for each $P\in\calP$.
\end{defn}

At first, it may seem that the acylindricity condition should not be included, as in the groups case it follows from the other conditions. However, it is needed to exclude an example of the following kind: just take any relatively hyperbolic space $Y$, cross it with $\R$ to obtain $X$ and consider the projection on the first factor $\pi$.
Also, the proof of the acylindricity condition in the groups case involves in a rather essential way the group action and the properness of a Cayley graph of $G$, so it seems likely that one needs to have extra conditions in the metric setting.

We make the observation that, in a natural sense, the definition we gave is invariant under quasi-isometries. If $\pi_i:X_i\to Y_i$, for $i=0,1$ are maps between metric spaces, we say that $f:X_0\to X_1$ is \emph{coarsely compatible} with $\pi_0,\pi_1$ if for each $y\in Y_0$ and $r\geq 0$, we have that $\pi_1\circ f(\pi_0^{-1}(B_r^{Y_0}(y)))$ has diameter bounded by $C=C(f,r)$.

\begin{prop}\label{qiinv}
 Suppose $\calP_1\hookrightarrow_h (X_1,\pi_1)$, where $\pi_1:X_1\to Y_1$ and let $\pi_0:X_0\to Y_0$ be coarsely Lipschitz and coarsely surjective. Suppose that $f:X_0\to X_1$ is a quasi-isometry so that both $f$ and any quasi-inverse of $f$ are compatible with $\pi_0,\pi_1$.
 
 Then there exists $D_0$ so that for each $D\geq D_0$ we have $\calP_0\hookrightarrow_h (X_0,\pi_0)$ for $\calP_0=\{f^{-1}(N_D(P))\}_{P\in\calP_1}$.
\end{prop}

\begin{proof}
 Consider a map $\overline{f}:Y_0\to Y_1$ mapping $y$ to any point in $\pi_1\circ f(\pi_0^{-1}(B_R^{Y_0}(y)))$, where $R$ is large enough that the image $N_R(\pi_0(X_0))=Y_0$. This map is easily seen to be coarsely Lipschitz. Also, it has a coarsely Lipschitz quasi-inverse, constructed in a similar way starting from a quasi-inverse of $f$.
 
 \begin{figure}[h]
 \includegraphics[scale=0.8]{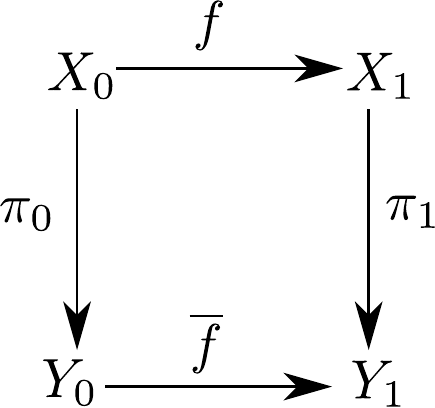}
 \caption{The diagram coarsely commutes.}
 \end{figure}

 In particular, $\overline{f}$ is a quasi-isometry. By quasi-isometry invariance of relative hyperbolicity, we see that $Y_0$ is hyperbolic relative to the collection of subsets $\calP_0=\{\pi_0(f^{-1}(N_D(P)))\}_{P\in\calP_1}$. This is because the Hausdorff distance between the image via $\overline{f}$ of $\pi_0(f^{-1}(N_D(P)))$ is within uniformly bounded Hausdorff distance of $\pi_1(P)$.
 
 The fact that $\pi_0|_{P_0}$ is proper for each $P_0\in\calP_0$ and the acylindricity condition are readily checked just chasing the definitions.
\end{proof}

The following corollary is immediate.

\begin{cor}
 Suppose $\calP_1\hookrightarrow_h (X_1,\pi_1)$ and let $f:X_0\to X_1$ be a quasi-isometry. Then there exists $D_0$ so that for each $D\geq D_0$ we have $\calP_0\hookrightarrow_h (X_0,\pi_1\circ f)$ for $\calP_0=\{f^{-1}(N_D(P))\}_{P\in\calP_1}$.
\end{cor}

We now proceed showing a naturally occurring example of hyperbolically embedded family in a non-relatively hyperbolic space (which is not a group).

\begin{thm}\label{teich}
 Let $Teich^{WP}(S_{g,p})$ be the Teichm\"uller space of the closed connected orientable surface of genus $g$ with $p$ punctures endowed with the Weil-Petersson metric $d_{WP}$, and suppose $3g+p\geq 6$. Let $\phi: S_g\to S_g$ be a pseudo-Anosov with axis $A$ in $Teich^{WP}(S_{g,p})$ and let $\calA$ be a collection of distinct axes of conjugates of $\phi$. Then
$$\calA\hookrightarrow_h (Teich^{WP}(S_{g,p}),\pi),$$
for $\pi: Teich^{WP}(S_{g,p})\to \calC(S_{g,p})$ the shortest-curve map to the curve complex $\calC(S_{g,p})$ of $S_{g,p}$.
\end{thm}

An analogous theorem holds for the pants graph, as it is quasi-isometric to $Teich^{WP}(S_{g,p})$ \cite{Br-pants}.

\begin{proof}
 As $\calC(S_{g,p})$ is hyperbolic \cite{MM1}, it is hyperbolic relative to $\calB=\{\pi(A)\}$ if (and only if) all elements of $\calB$ are uniformly quasi-convex and for each $D$ there exists $B$ so that all distinct $B_1,B_2\in\calB$ have the property that the diameter of $N_D(B_1)\cap B_2$ is at most $B$, see \cite[Section 7]{Bow-99-rel-hyp}. This is the case because pseudo-Anosovs act loxodromically on the curve complex \cite{MM1} and the action of a mapping class group on the corresponding curve complex is acylindrical \cite{B-tight} (this is the same reason why $\phi$ is contained in a virtually cyclic hyperbolically embedded subgroup of the mapping class group of $S_{g,p}$, \cite[Theorem 6.47]{DGO}).

Also, the fact that $\phi$ acts loxodromically on the curve complex also implies that $\pi$ is proper when restricted to $A$ or one of its translates.

In order to show the geometric separation property, we can use the following results from \cite{MM2} about Teichm\"uller space (see also the discussion before and after \cite[Theorem 2.6]{BMM-WP2}). We write $A\approx_{K,C} B$ if the quantities $A,B$ satisfy
$$A/K-C\leq B\leq KA+C.$$
Let $\{\{A\}\}_L$ denote $A$ if $A\geq L$ and $0$ otherwise.

\begin{thm}
For $Y$ a non-annular subsurface of $S_{g,p}$ and for $x,y\in Teich^{WP}(S_{g,p})$, denote by $d_{\calC(Y)}(x,y)$ the diameter in the curve complex $\calC(Y)$ of $Y$ of $\pi_Y(x)\cup\pi_Y(y)$, where $\pi_Y$ denotes the subsurface projection on $\calC(Y)$.
\begin{enumerate}
 \item (Distance Formula) There exists $L_0$ with the property that for each $L\geq L_0$ there are $K,C$ so that, for each $x,y\in Teich^{WP}(S_{g,p})$,
 $$d_{WP}(x,y)\approx_{K,C} \sum_Y \{\{d_Y(x,y)\}\}_L,$$
 where the sum is taken over all (isotopy classes of) non-annular subsurfaces $Y$ of $S$.
 \item (Bounded Geodesic Image Theorem)  There exists $C$ with the following property. If $\gamma$ is a geodesic in $\calC(S)$ so that for some proper subsurface $Y$ we have that $\pi_Y(v)$ is non-empty for every vertex $v\in\gamma$ then $\pi_Y(\gamma)$ has diameter at most $C$.
\end{enumerate}
\end{thm}

Suppose that $\pi(x),\pi(y)\in \pi(A)$ are far enough compared to $r$. Pick $p,q$ in Teichm\"uller space so that $\pi(p)\in B_r^{\calC(S_{g,p})}(\pi(x))$, $\pi(q)\in B_r^{\calC(S_{g,p})}(\pi(y))$, and suppose $d_{WP}(p,q)\leq D$. There is a bound $C_1=C_1(D, S_{g,p})$, coming from the distance formula, on $d_{\calC(Y)}(\pi_Y(p),\pi_Y(q))$ for any subsurface $Y\subseteq S_{g,p}$. By the Bounded Geodesic Image theorem (up to increasing $C_1$) we know that any proper subsurface $Y$ so that $d_{\calC(Y)}(p,x)\geq C_1$ appears along any geodesic from $\pi(p)$ to $\pi(x)$. For any such subsurface $Y$, as any geodesic from $\pi(q)$ to $\pi(y)$ is far from $\partial Y$ in $\calC(S_{g,p})$, using the Bounded Geodesic Image theorem we can make the estimate
$$d_{\calC(Y)}(x,p)\leq d_{\calC(Y)}(x,y)+d_{\calC(Y)}(y,q)+d_{\calC(Y)}(q,y)\leq d_{\calC(Y)}(x,y)+ C_2,$$
where $C_2$ depends on $C_1$ and $S_{g,p}$ only.
In particular, using the distance formula with threshold $L>\max\{d_{\calC(Y)}(x,y)+ C_2\}$, we get a bound on $d_{WP}(x,p)$, and the proof is complete.
\end{proof}

\bibliographystyle{alpha}
\bibliography{phd.bib}

\end{document}